\documentclass[12pt]{article}

\usepackage[english]{babel}
\usepackage[margin=1in]{geometry}

\usepackage[utf8]{inputenc}
\usepackage{amsmath}
\usepackage{amsthm}
\usepackage{amssymb}
\usepackage{amsbsy}
\usepackage{graphicx}
\usepackage[colorinlistoftodos]{todonotes}
\usepackage{hyperref} 
\usepackage{color}
\usepackage{pslatex}
\usepackage{fixmath}
\usepackage{bm}
\usepackage{enumerate}
\mathchardef\mhyphen="2D
\usepackage{setspace}

\newtheorem{theorem}{Theorem}
\newtheorem{lemma}[theorem]{Lemma}

\newtheorem{conjecture}[theorem]{Conjecture}
\theoremstyle{definition}
\newtheorem{definition}{Definition}

\newtheorem{strategy}{Strategy}

\title{Counting Counterfeit Coins: A New Coin Weighing Problem}
\author{Nicholas Diaco}
\date{}

\begin{document}
\maketitle

\begin{abstract}
\normalsize

In 2007, a new variety of the well-known problem of identifying a counterfeit coin using a balance scale was introduced in the sixth International Kolmogorov Math Tournament. This paper offers a comprehensive overview of this new problem by presenting it in the context of the traditional coin weighing puzzle and then explaining what makes the new problem mathematically unique. Two weighing strategies described previously are used to derive lower bounds for the optimal number of admissible situations for given parameters. Additionally, a new weighing procedure is described that can be adapted to provide a solution for a broad spectrum of initial parameters by representing the number of counterfeit coins as a linear combination of positive integers. In closing, we offer a new form of the traditional counterfeit coin problem and provide a lower bound for the number of weighings necessary to solve it.

\end{abstract}

\pagebreak 

\section{Introduction}\label{int}

The problem of identifying a single counterfeit coin in a set of ordinary coins using the fewest possible number of measurements on a balance beam is often thought to be folklore. It turns out that the original problem is fairly modern, with the first records of its existence dating back to around 1945 when Grossman posed the following question in \cite{1945}:

\begin{quote}
Given 12 coins, at least 11 of which have the same weight, how can one be guaranteed after three measurements on a balance scale to either isolate the defective coin and find its weight relative to the other coins or prove its nonexistance?
\end{quote}

This problem became an instant classic in the mathematical world. Numerous generalizations of the ``one counterfeit coin problem'' exist, with some versions of the problem including more real coins, adaptive and non-adaptive weighing schemes, and balance beams with more than two pans, just to name a few. A fairly thorough overview of these recent additions and their solutions is offered in \cite{2pan1,2pan2,npan}. As it turns out, the problem of finding counterfeit coins turns out to be more than just a simple puzzle. It is perhaps no coincidence that Claude Shannon introduced information theory, one of the last century's most significant contributions to modern society, in his renowned paper \cite{shannon} just three years after the introduction of the coin weighing problem --- solutions to the balance puzzle are intimately connected with the construction of information theoretic error correcting codes. Among many other applications, weighing coins can even make finding carriers of blood-borne diseases easier; it is shown in \cite{blood} that a pooling method adapted from the counterfeit coin problem can find infected individuals more efficiently than traditional, one-person-at-a-time testing can.

Naturally, the next step forward seemed to be considering the problem of finding multiple fake coins. What might have been a simple progression turned out to be a huge challenge, however; mathematicians struggled with the addition of just one more fake coin, let alone several. One notable theorem of Pyber in \cite{pyber} in 1986 was as follows: If exactly $m$ (lighter) counterfeit coins are to be found among $n$ coins then the counterfeit coins can be found in at most $$ \log_3 \dbinom{n}{m}+15m $$ steps in all cases. This bound is slightly improved in \cite{pyberrekt} and an upper bound for the problem in which the number of fake coins is unknown is shown in \cite{pyberrekt2}, but further improvement has been hard to come by. New attempts at finding better bounds for various counterfeit coin problems have pulled ingenious techniques from graph theory \cite{purdy1}, sequential algorithms \cite{4fake}, brute force dynamic programming \cite{analytic}, and even geometry \cite{geometry}, yet very few of these tactics have offered much new insight on the general many-fake-coin problem. For example, whether or not it is possible to achieve the information theoretic lower bound for locating any given number of fake coins is still an open question with no end in sight in the near future. A brief summary of the counterfeit coin problem can be found in \cite{GN} along with many related open problems that still have not been solved today, twenty years after the review's publication.

In 2007, an unusual coin weighing problem was suggested by Alexander Shapovalov for the sixth International Kolmogorov Math Tournament \cite{TurKolm}. It was unlike any of the aforementioned coin weighing problems --- rather, it seemed like a converse of the traditional question:

\begin{quote}
You have 80 coins that are identical in appearance. Among them are three fake coins. The genuine coins all have the same weight and the fake coins all have the same weight, but the fake coins are lighter than the real ones. In addition, you know the location of each of the fake coins.

Your friend knows that there are either three or two fake coins in the pile. Without revealing the identity of any of the 80 coins, how can you use a beam balance to convince your friend that there are exactly three fake coins?
\end{quote}

Konstantin Knop offered several interesting solutions to this problem in the case of 100 total coins in \cite{Knop} (in Russian). One of Knop's solutions solves the problem in a mere three steps; for comparison to the traditional problem, Pyber's upper bound from \cite{pyber} guarantees that one would be able to find the three counterfeit coins in at most 56 steps. Tanya Khovanova later published a short blog post \cite{TK} on this puzzle, extending the problem to have another goal: minimizing the amount of information revealed during the weighing process. Even though the identity of no particular coin is revealed, some information about the distribution of the weights is inherently forgone during weighings. To this end, \cite{TK} introduced a revealing coefficient, $R$, as a metric to compare the relative efficiency of different strategies --- the lower the value of the revealing coefficient, the lower the information loss, and the better the strategy. A paper coauthored by Khovanova and Diaco, the author, explores the problem in greater detail \cite{NDTK}. Some advances include analysis and generalization of several strategies, two proofs of optimality for select cases, and introduction of the revealing factor $X$, an alternative to the revealing coefficient, for comparing weighing schemes. Additionally, a distinction is made between two broad classes of strategies: discreet and indiscreet.

In general, the problem can be stated as follows:

\begin{quote}
You have $t$ \textbf{t}otal coins that are identical in appearance. Among them are $f$ \textbf{f}ake coins. The genuine coins all have the same weight and the fake coins all have the same weight, but the fake coins are lighter than the real ones. In addition, you know the location of each of the fake coins.

Your friend knows that there are either $f$ or $d$ (the number we are trying to \textbf{d}isprove) fake coins in the pile. Without revealing the identity of any of the $t$ coins, how can you use a beam balance to convince your friend that there are exactly $f$ fake coins?
\end{quote}

This paper effectively serves as a continuation of \cite{NDTK}. After introducing several preliminary definitions and solutions to the original problem, we will provide a more formal mathematical definition of our new coin weighing problem, and use it to reinterpret both the revealing factor and what it means to have a solution for a given set of parameters $t$, $f$, and $d$. We will slightly improve upon two previously generalized strategies for solving the problem and derive lower bounds on remaining information that any optimal weighing strategy must satisfy. A new discreet strategy is described that can represent the number of fake coins as a linear combination of positive integers, offering a marked improvement over several previously described strategies by requiring fewer weighings, revealing less information, and generating solutions for a much broader range of parameters.

We will list a few conjectures and areas of future study and make suggestions for how they should be approached. We will also briefly highlight several possible applications of this problem in areas such as cryptology and information science. In conclusion, we will create a new version of the traditional coin weighing problem motivated by Shapovalov's puzzle and its solutions, using a result from \cite{purdy1} to derive a lower bound for the number of weighings necessary to solve it.

\section{Preliminary Definitions and Results}\label{sec:1}

The following is the official solution offered by the Kolmogorov Math Tournament \cite{TurKolm}:

\begin{strategy}\label{str:1} Divide all the coins into 5 piles: $A$ and $B$ with 10 coins each, and $C$, $D$, and $E$ with 20 coins each. Place one fake coin in each of $A$, $D$, and $E$. For the first weighing, compare $A + C$ against $B + D$ on the scale. Next, place $A + B$ and $E$ on opposite sides of the scale. Finally, compare $C + D$ against $A + B + E$. The scale will be balanced during the first two weighings, and the pan containing $C$ and $D$ will tilt downwards during the final weighing.
\end{strategy}

This strategy proves to an observer that it is impossible for there to be two fake coins. Furthermore, these weighings show that there are two possibilities: either piles A, D and E each contain one fake coin, or piles B, C, and E each contain one fake coin. In either case, no coin has its identity revealed. The total number of different ways in which the $f$ fake coins can be distributed is easily calculated as $|A|\cdot |D| \cdot |E| + |B| \cdot |C| \cdot |E| = 10\cdot20\cdot 20 + 10\cdot20\cdot 20 = 8000$. Before the weighings take place, however, the $f$ fake coins can be distributed in a total of $\binom{80}{3} = 82160$ equally probable, yet unique ways in the eyes of the observer. Clearly, the number of possible distributions is reduced significantly.

As defined previously in \cite{NDTK}, we would like to introduce the notion of a revealing factor to quantify this observation. If the observer knows that there are exactly $f$ fake coins before the weighings take place, then any one of $\binom{t}{f}$ possible arrangements of these $f$ coins is possible; this value will be referred to as \textit{\# old possibilities}. The number of possible distributions of the $f$ coins after the weighings is denoted by \textit{\# new possibilities}.

\begin{definition}
The \textit{revealing factor} of some successful strategy, denoted by $X$, is defined as

$$X = \frac{\text{\# old possibilities}}{\text{\# new possibilities}}.$$
\end{definition}

A lower value of $X$, which correlates to a larger number of new possibilities, is clearly preferable. For Strategy~\ref{str:1}, we have $X=10.27$. As suggested by Khovanova in \cite{TK} but not used in this paper, another means of quantifying the amount of information lost to an observer during weighings is the revealing coefficient $R$, which can be calculated as $1-1/X$. Strategy~\ref{str:1} is revisited in section~\ref{gensha}.

\begin{strategy}\label{str:2} Split all the coins into three piles of size 26 and one smaller pile of size 2, with one fake coin in each of the larger piles. First, show that the three groups of 26 coins balance each other on the scale. Next, compare one of the real coins in the smaller pile to a single real coin from one of the larger piles.
\end{strategy}

Once again, an observer of this strategy should be convinced that there are three fake coins as opposed to two. However, three coins were shown to be real in the process, breaking the rules of Shapovalov's original problem. Despite this fact, we have $ \text{ \# new possibilities} = 26\cdot 26 \cdot 25 = 16900 $, and $X\approx 4.86$. This strategy, perhaps unexpectedly, does a much better job at keeping the fake coins hidden than does Strategy~\ref{str:1}. In order to allow strategies such as this one to be accepted, we introduced two different classes of solutions in \cite{NDTK}:

\begin{definition}
A set of weighings or a strategy for which the identity of no particular coin is revealed is a \textit{discreet} strategy. Otherwise, we call the strategy \textit{indiscreet}.
\end{definition}

As proven in \cite{NDTK}, constructing a discreet weighing strategy is not possible in all cases:

\begin{lemma}
For a strategy to be discreet, it is necessary that $1 < f < t-1$.
\end{lemma}

On the other hand, it turns out that indiscreet strategies are quite often (and quite counterintuitively) significantly less revealing than their discreet counterparts. This strategy is revisited and generalized in section~\ref{indiscreet}.

\begin{strategy}\label{str:3} Divide all the coins into nine piles: $A_1$, $B_1$, $C_1$, $A_2$, $B_2$, $C_2$, $A_3$, $B_3$ and $C_3$ of sizes 24, 1, 2, 24, 1, 2, 23, 2, and 1, respectively. Demonstrate that $A_1+B_1$,  $A_2+B_2$ and $A_3+B_3$ all have the same weight. Additionally, show that $B_1+C_1$, $B_2+C_2$, and $B_3+C_3$ all balance each other on the scale.
\end{strategy}

An observer can conclude that only the following distributions of fake coins are possible:

\begin{enumerate}
\item one fake coin in one of each: $A_1$, $A_2$, $A_3$ (sizes 24, 24, 23).
\item one fake coin in one of each: $B_1$, $B_2$, $B_3$ (sizes 1, 1, 2).
\item one fake coin in one of each: $C_1$, $C_2$, $C_3$ (sizes 2, 2, 1).
\end{enumerate}

In each case we have ruled out the possibility of there being two fake coins, and the strategy is discreet. The number of ways for the $f$ fake coins to be distributed after the weighings is $\text{ \# new possibilities} = 24\cdot 24 \cdot 23 + 1\cdot 1 \cdot 2 + 2 \cdot 2 \cdot 1 = 13254$, so that $X \approx 6.20$. This strategy is discussed in \cite{Knop,NDTK} and is revisited and generalized in section~\ref{discreet}.

The following relatively trivial strategy does not apply to Shapovalov's original puzzle, but is of great importance in general. If some integer $a>1$ divides both $f$ and $t$ but not $d$, then the following strategy is discreet:

\begin{strategy}\label{str:6}
Divide all the coins into $a$ piles, each with an equal number of fake coins. Proceed by comparing all of these piles with each other on the scale.
\end{strategy}

Since all of the weighings will be balanced, this strategy simply proves that the number of fake coins is divisible by $a$. In general, the smaller the value of $a$, the lower the revealing factor, which can be calculated exactly as
\begin{equation}\label{xinf}
X =\frac{\binom{t}{f}}{\binom{t/a}{f/a}^a} \sim \dfrac{f^{f}}{f!} \bigg( \dfrac{(\frac{f}{a})!}{(\frac{f}{a})^{(\frac{f}{a})}} \bigg) ^a,
\end{equation}
where the right hand side is the value that $X$ approaches as $t$ tends to infinity. It was shown in \cite{NDTK} that this strategy is optimal for $f=2$, $t$ even, and $d$ odd, with the choice of $a=2$.

Several of these strategies for the case of $t=100$, $f=3$, and $d=2$, in addition to more examples of both insufficient and correct solutions to the original problem, can be found in Knop's article \cite{Knop} (in Russian).

\subsection{The Generalized Original Counterfeit Coin Problem}\label{genorig}

We will now offer a technical overview of the traditional counterfeit coin problem. The definitions, notations, and conventions used in the remainder of Section 2.1 are taken from \cite{general} and adapted for the purpose of this paper:

Assume that we have $t$ total objects, at most $f$ of which are defective. Our goal is to create a weighing strategy that exactly locates the $f$ defective objects in $m$ weighings on a balance scale. Let $\mathbb{R}^t$ be the $t$-dimensional Euclidean space, $\mathbold{a} \cdot \mathbold{b} $ be the inner product of vectors $\mathbold{a}$ and $\mathbold{b}$ from $\mathbb{R}^t$, and let $\mathbf{1} = (1,1, \ldots ,1)$ be the vector of length $t$ with $1$ as every element. The cardinality of a set $E \subseteq \mathbb{R}^t$ is denoted by $|E|$. We denote the set of all sequences of length $t$ over the alphabet $I = \{ -1, 0, 1 \}$ by $I^t$; equivalently, $I^t = \{-1, 0, 1\}^t$.

We are given $t$ objects which are each described by one of two positive weights $w_0$ or $w_1$, such that $w_0<w_1$. We will let the standard weight be $w_1$, and consider $w_0$ to be non-standard, or counterfeit, weight; the actual numerical values of $w_0$ and $w_1$ do not matter as shown in \cite{purdy1}. We can then describe the set of $t$ objects by a vector $\mathbf{x} = (x_1, \ldots, x_t) \in \{0,1 \}^t$, where $x_i = 1$ if the weight of the object corresponding to $x_i$ is $w_1$, and $x_i=0$ if the object corresponding to $x_i$ has weight $w_0$.

For every \textit{weighing}, each object is assigned a value of $h_i \in I$. This value gives the object's location in a particular weighing: if $h_i=0$, the object does not participate in the weighing; if $h_i=-1$ the object is located on the left pan, and similarly if $h_i=1$ the object is located on the right pan. In this way, a weighing can be given by a non-zero vector $\mathbf{h} = (h_1,\ldots, h_t) \in I^t$. Both pans must have the same number of coins for each weighing; \textit{i.e.}, $\mathbf{h} \cdot \mathbf{1} = 0$. The result of a weighing is determined by the value $s(\mathbf{x},\mathbf{h}) = \textrm{sign}(\mathbf{x} \cdot \mathbf{h})$: the weighing is balanced for $s(\mathbf{x},\mathbf{h}) = 0$, the right pan outweighs the left when $s(\mathbf{x},\mathbf{h}) = 1$, and the left pan outweighs the right when $s(\mathbf{x},\mathbf{h}) = -1$. 

We will denote the initial set of admissible distributions of weights of objects, or the \textit{set of admissible situations}, by $Z \subseteq I^t$. Furthermore, each \textit{admissible situation} is given by an element $\mathbf{z} \in Z$. An admissible situation is one of potentially many ways for the non-standard objects to be distributed amongst the $t$ total objects; in other words, each $\mathbf{z}$ represents a possibility for the true distribution of the objects, $\mathbf{x}$. This notion of admissible situations offers a quantitative means of describing the lack of initial information about the objects' weights.

Following some weighing $\mathbf{h}$, the set $I^t$ is partitioned by the plane $[\mathbf{x},\mathbf{h}] = 0$ into three disjoint sets $W(s|I^t,\mathbf{h}) = \{ x \in I^t | s(\mathbf{x},\mathbf{h}) =s \}$, $s \in I$. This additionally corresponds to a partition of set $Z$ into the disjoint sets $W(s|Z,\mathbf{h}) = W(s|I^t,\mathbf{h}) \cap Z$, $s \in I$. We say that a weighing $\mathbf{h}$ \textit{classifies} the elements $\mathbf{z} \in Z$ according to the subsets of $Z = W(0|Z,\mathbf{h})+ W(1|Z,\mathbf{h})+ W(-1|Z,\mathbf{h})$.

\begin{definition}
\textit{A weighing strategy $\mathcal{A}$ of length m}, where $m$ denotes the number of weighings, is a sequence of consecutive weighings that follows a predetermined set of instructions. A WS $\mathcal{A}$ checks a situation $\mathbf{z} \in Z$ by performing these weighings on the set of $t$ objects.
\end{definition}

The \textit{procedure of the WS $\mathcal{A}$ in the situation $\mathbf{z}$} is denoted by the sequence of weighings $\mathcal{A}(\mathbf{z}) = \langle \mathbf{h}^1, \ldots, \mathbf{h}^m \rangle$, where $\mathbf{h}^j \in I^t$, is the $j$th weighing for $1 \le j \le m$. Each weighing $\mathbf{h}^j$ may be dependent on the result of the previous weighing, $s(\mathbf{z},\mathbf{h}^{j-1})$, or may simply be determined before any of the weighings take place (adaptive and oblivious weighings, respectively). For all such $\mathcal{A}$, an initial weighing $\mathbf{h}^1$ is given. In addition, the sequence $\mathbf{s}(\mathbf{z}|\mathcal{A}) = (s(\mathbf{z},\mathbf{h}^1), \ldots, s(\mathbf{z},\mathbf{h}^m))$ of results found by $\mathcal{A}(\mathbf{z})$ is called the \textit{syndrome of a situation $\mathbf{z}$}. Conversely, a sequence $(s_1, \ldots , s_m) \in I^m$ which is the syndrome of some situation $\mathbf{z} \in Z$ is called a $(Z, \mathcal{A})$\textit{-syndrome}.

We define $W(\mathbf{s}|\mathcal{A}) \subseteq I^t$ as $W(\mathbf{s}|\mathcal{A}) = \{ \mathbf{z} \in I^t | \mathbf{s}(\mathbf{z}|\mathcal{A}) = \mathbf{s} \}$; $W(\mathbf{s}|Z, \mathcal{A}) = W( \mathbf{s} | \mathcal{A} ) \cap Z$, or the set of all situations with the same syndrome $\mathbf{s}$. Additionally, we denote the set of all $(Z, \mathcal{A})$-syndromes as $S(Z, \mathcal{A})$. For two distinct syndromes $\mathbf{s}^1 \not= \mathbf{s}^2$, we have $W(\mathbf{s}^1|\mathcal{A}) \cap W(\mathbf{s}^2|\mathcal{A}) = \varnothing$ by definition; this means that the set $S(Z,\mathcal{A})$ induces a partition of $Z$ into the subsets $W(\mathbf{s}|\mathcal{A}) \subseteq Z$, $\mathbf{s} \in S(Z, \mathcal{A})$. Consequently, we can say that the strategy $\mathcal{A}$ classifies the situations from $Z$ into subsets $W(\mathbf{s} | \mathcal{A}) \in S(Z,\mathcal{A})$.

\begin{definition} 
A WS $\mathcal{A}$ is said to \textit{identify the situations in a set $Z$} if the condition $|W(\mathbf{s}|Z,\mathcal{A})| = 1$ is satisfied for all $\mathbf{s} \in S(Z, \mathcal{A})$.
\end{definition}

The above definition states that when a weighing strategy $\mathcal{A}$ generates a $(Z,\mathcal{A})$-syndrome that corresponds with exactly one $\mathbf{z} \in Z$, it identifies the situation in a set $Z$ and thus solves the problem of exactly locating the non-standard objects. Alternatively, we can say that $\mathcal{A}$ identifies the situations in a set $Z$ if it has found $\mathbf{z} \in Z$ representing the true weight distribution $\mathbf{x}$ of the $t$ objects.

\section{Generalizing Shapovalov's Counterfeit Coin Problem}\label{gensha}

Recall the explanation of Shapovalov's problem given in Section \ref{int}. We have $t$ total coins and $f$ fake coins, and our goal is to prove that there are exactly $f$ --- not $d$ --- fake coins, without revealing any of the coins' identities. In this section it is assumed that all calculations are for the party observing the weighings take place rather than for the one performing them, since for the individual carrying out the weighings it is clear that $|Z|=1$. We define $Z_f$ as the set of initial admissible situations in which there are exactly $f$ fake coins, and $Z_d$ is the set of admissible situations in which there are exactly $d$ fake coins. It is easy to see that $|Z_n|= \dbinom{t}{n}$ for any $0 \le n \le t$.

\begin{definition}\label{def:3}
A WS $\mathcal{A}$ is said to be a \textit{successful strategy}, or prove that there are $f$ and not $d$ fake coins, if the following conditions are met:

\begin{onehalfspacing}

\begin{enumerate}
\item $|W(\mathbf{s}|Z_d,\mathcal{A})| = 0$ is satisfied for $\mathbf{s} \in S(Z_d, \mathcal{A})$.
\item $|W(\mathbf{s}|Z_f,\mathcal{A})| > 0$ is satisfied for $\mathbf{s} \in S(Z_f, \mathcal{A})$.
\end{enumerate}

\end{onehalfspacing}

\end{definition}

\begin{definition}\label{def:4}
A successful strategy given by a WS $\mathcal{A}$ is said to be \textit{discreet} if the $i$th element $z_i$ of all admissible situations $ \mathbf{z} \in W(\mathbf{s}|Z_f, \mathcal{A})$ is equal to $0$ and $1$ at least one time each for $1 \le i \le t$; in this way, it is not possible to claim that any specific coin is certainly counterfeit or real. A successful strategy that does not meet this requirement is said to be \textit{indiscreet}.
\end{definition}

\begin{theorem}\label{bounds}
For any successful discreet weighing strategy $\mathcal{A}$, 
\begin{equation}\label{eq:bound}
\max \Bigl( \Bigl\lceil \frac{t}{f} \Bigr\rceil, \Bigl\lceil \frac{t}{t-f} \Bigr\rceil \Bigr) \le |W(\mathbf{s}|Z_f,\mathcal{A})| < \dbinom{t}{f}.
\end{equation}

\end{theorem}

\begin{proof}
LHS: As shown in definition~\ref{def:4}, each element $z_i$ corresponding to a unique coin for all $\mathbf{z} \in W(\mathbf{s}|Z_f, \mathcal{A})$ must be $0$ at least once, and $1$ at least once for a successful strategy $\mathcal{A}$ to be discreet. In order for each element $z_i$ to assume the value $0$ once, a minimum of $\lceil t/f \rceil$ unique $\mathbf{z} \in W(\mathbf{s}|Z_f, \mathcal{A})$ is necessary. By symmetry, there must be at the very least $\lceil t/(t-f) \rceil$ unique admissible situations $\mathbf{z}$ for each element $z_i$ to assume the value $1$ once.

RHS: If it were possible to prove to an observer that there are exactly $f$ counterfeit coins without revealing any additional information, there would be $|Z_f|$ possible ways for the fake coins to be distributed. However, as shown in \cite{NDTK}, if a WS $\mathcal{A}$ proves that there are $f$ and not $d$ fake coins then the number of admissible situations is necessarily reduced; consequently, $W(\mathbf{s}|Z_f,\mathcal{A}) \subset Z_f$ and thus $|W(\mathbf{s}|Z_f,\mathcal{A})| < |Z_f|$.
\end{proof}

This generalization of Shapovalov's puzzle allows us to redefine the revealing factor more formally:

\begin{definition}\label{def:5}
The revealing factor $X$ for a given set of initial parameters $t$, $f$, and $d$ and a WS $\mathcal{A}$ is defined as the ratio of the number of possibilities for the distribution of the $f$ fake coins before the weighings take place to the number of possibilities after the weighings:

$$X = \dfrac{|Z_f|}{|W(\mathbf{s}|Z_f,\mathcal{A})|} $$
\end{definition}

The calculation of $W (\mathbf{s}|Z_f,\mathbf{h}^j)$ at any given step $1 \le j \le m$ during a series of weighings must be done on a case-by-case basis, but explicit formulas can be given for the initial weighing. Given that the observer knows that there are $f$ counterfeit coins, the number of new possibilities after one balanced weighing (\textit{i.e.} $m=1$) in which there are $n$ coins in both pans is

\begin{equation}\label{eq:1}
|W(0|Z_f,\mathbf{h}^1)| =
\sum_{i=0}^{n} \dbinom{n}{i}^2 \dbinom{t-2n}{f-2i} = B_{(f,n)},
\end{equation}

and for a single unbalanced weighing the number of new possibilities is 

\begin{equation}\label{eq:2}
|W(-1|Z_f,\mathbf{h}^1)| = |W(1|Z_f,\mathbf{h}^1)| =
\sum_{j=i+1}^{n} \sum_{i=0}^{n-1} \dbinom{n}{i} \dbinom{n}{j} \dbinom{t-2n}{f-i-j} = U_{(f,n)}.
\end{equation}

Because $W(s|Z_f,\mathbf{h}^1)$, $s \in I$ are disjoint partitions of $Z_f$ according to our definition of classification, the above equations lead to a special case of Vandermonde's identity: 
\begin{equation}\label{eq:vandermonde}
B_{(f,n)} + 2\cdot U_{(f,n)} = \sum_{j=0}^{n} \sum_{i=0}^{n} \dbinom{n}{i} \dbinom{n}{j} \dbinom{t-2n}{f-i-j} = \binom{t}{f}.
\end{equation}

Additionally, using equations (\ref{eq:1}) and (\ref{eq:2}), we can derive a tighter upper bound for $|W(\mathbf{s}|Z_f,\mathcal{A})|$ than the one given in (\ref{bounds}). Since every successful discreet WS $\mathcal{A}$ must have a positive number of weighings $m \ge 1$, the value of $|W(\mathbf{s}|Z_f,\mathcal{A})|$ must be less than or equal to the maximum value of possible distributions of the $f$ counterfeit coins after one discreet weighing, denoted by $\overline{\text{max}}(B_{(f,n)}, U_{(f,n)})$. That is,

\begin{equation}\label{newbounds}
 |W(\mathbf{s}|Z_f,\mathcal{A})| \le \overline{\text{max}}(B_{(f,n)}, U_{(f,n)}).
\end{equation}

For instance, consider the case with $f=2$, $d$ odd, and $t \ge 8$. Since any unbalanced weighing is necessarily indiscreet as shown in \cite{NDTK}, we have $\overline{\text{max}}(B_{(2,n)}, U_{(2,n)}) = \overline{\text{max}}(B_{(2,n)})$. For the given parameters, this maximum value occurs at $n=1$, so $\overline{\text{max}}(B_{(2,n)}) = B_{(2,1)} = \binom{t-2}{2}+1$, which is consistent with the optimal weighing strategies described for the given parameters in \cite{NDTK}.

In many coin weighing problems, great emphasis is placed on the order in which specified measurements take place. On the contrary, for any successful weighing strategy $\mathcal{A}$ of length $m$, the order in which the weighings $\mathbf{h}^j$ take place for $1 \le j \le m$ does not affect $W(\mathbf{s}|Z_f,\mathcal{A})$ or $W(\mathbf{s}|Z_d,\mathcal{A})$. Consider the procedure of a WS $\mathcal{A}$ in the situation $\mathbf{z}$: $\mathcal{A}(\mathbf{z}) = \langle \mathbf{h}^1, \ldots, \mathbf{h}^m \rangle$. Recall that the result $s(\mathbf{z},\mathbf{h}^j)$ of a weighing $\mathbf{h}^j$ classifies the elements $\mathbf{z} \in Z$ into one of three partition subsets $W(s(\mathbf{z},\mathbf{h}^j)|Z,\mathbf{h}^j) = W(s(\mathbf{z},\mathbf{h}^j)|I^t,\mathbf{h}^j) \cap Z$. The result $s(\mathbf{z},\mathbf{h}^k)$ of a subsequent weighing $\mathbf{h}^k$ for $k \neq j$ further classifies the elements $\mathbf{z} \in Z$ into $W((s(\mathbf{z},\mathbf{h}^j), s(\mathbf{z},\mathbf{h}^k))|Z, \langle \mathbf{h}^j, \mathbf{h}^k \rangle) = W(s(\mathbf{z},\mathbf{h}^j)|Z,\mathbf{h}^j) \cap W(s(\mathbf{z},\mathbf{h}^k)|Z,\mathbf{h}^k)$. Because this process continues in the same way for all subsequent weighings, given $\mathcal{A}(\mathbf{z}) = \langle \mathbf{h}^1, \ldots, \mathbf{h}^m \rangle$ we have

$$W(\mathbf{s}|Z,\mathcal{A}) = \bigcap_{j=1}^{m} W(s(\mathbf{z},\mathbf{h}^j)|Z,\mathbf{h}^j). $$

Since the intersection operation on sets is commutative, \textit{i.e.} $A \cap B = B \cap A$ for any sets $A$ and $B$, it follows that the order in which the weighings of a WS $\mathcal{A}$ take place do not affect $W(\mathbf{s}|Z,\mathcal{A})$.

Although the order of the sequence of weighings does not affect the final outcome, each weighing tends to reveal a different amount of information about the nature of the coins' distribution. If the goal of another party is to guess the locations of the $f$ fake coins, they may choose to do so before all the weighings have taken place. With this in mind, we may be able to specifically choose the order of the weighings to maximize the number of new possibilities at each given step. Consider Strategy~\ref{str:1}, where $\mathbf{h}^n$ denotes the $n$th listed weighing, and $\varnothing$ denotes the state before any weighings have taken place, \textit{i.e.} $W(\mathbf{s}|Z,\varnothing) = Z$.

\begin{table}[h!]
\centering
 \begin{tabular}{|c | c | c|} 
 \hline
 Weighings ($\langle\mathbf{h}\rangle$) & $|W(\mathbf{s}|Z_3,\langle\mathbf{h}\rangle)|$ & $|W(\mathbf{s}|Z_2,\langle\mathbf{h}\rangle)|$ \\ [0.5ex] 
 \hline
 $\varnothing$ & 82160 & 3160  \\ 
 \hline
 $\langle\mathbf{h}^1\rangle$ & 19140 & 1090  \\
 \hline
 $\langle\mathbf{h}^2\rangle$ & 25880 & 1180  \\
 \hline
 $\langle\mathbf{h}^3\rangle$ & 41080 & 780  \\
 \hline
 $\langle\mathbf{h}^1$, $\mathbf{h}^2\rangle$ & 16000 & 400  \\
 \hline
 $\langle\mathbf{h}^1$, $\mathbf{h}^3\rangle$ & 20000 & 290  \\
 \hline
 $\langle\mathbf{h}^2$, $\mathbf{h}^3\rangle$ & 16000 & 400  \\
 \hline
 $\langle\mathbf{h}^1$, $\mathbf{h}^2$, $\mathbf{h}^3\rangle$ & 8000 & 0  \\
 \hline
 \end{tabular}
 \caption{A list of the values of $|W(\mathbf{s}|Z,\langle\mathbf{h}\rangle)|$ for different combinations of the Strategy~\ref{str:1} weighings.}
 \label{table:1}
\end{table}

If an additional goal of our weighing strategy is to maximize $|W(\mathbf{s}|Z_3,\langle\mathbf{h}\rangle)|$ at each given step, the best sequence of weighings in this case is clearly $\mathbf{h}^3$, $\mathbf{h}^1$, $\mathbf{h}^2$.

\section{Optimal Indiscreet Strategies: A Lower Bound on $|W(\mathbf{s}|Z_f, \mathcal{A})|$}\label{indiscreet}

We define an \textit{optimal} strategy as the weighing strategy $\mathcal{A}$ which yields the highest possible value of $|W(\mathbf{s}|Z_f, \mathcal{A})|$ for given parameters. As previously shown in \cite{NDTK}, the following is a generalization of Strategy~\ref{str:2}. This strategy is successful and indiscreet for some positive integer $a>1$ such that $a \mid f$, $a \nmid d$, $a \nmid t$, $t > 2a$, and $\big\lfloor \frac{t}{a} \big\rfloor - \big\lceil \frac{d}{a} \big\rceil > \frac{f}{a}$:

\begin{quote}
Divide the $t$ coins into $a$ piles of size $\big\lfloor \frac{t}{a} \big\rfloor $ containing $\frac{f}{a}$ fake coins each, along with an additional pile of $t - a \big\lfloor \frac{t}{a} \big\rfloor$ leftover real coins. Use the balance beam to show that each of the $a$ piles of size $\big\lfloor \frac{t}{a} \big\rfloor$ has the same weight. Next, compare all of the leftover coins to each other so show that they are equal in weight. At this point, the number of fake coins may thus be expressed in one of two ways: $ak$ or $ak + t - a \big\lfloor \frac{t}{a} \big\rfloor$ for some nonnegative integer $k$. If necessary, compare a coin from the leftover pile to several real coins from the larger piles to demonstrate that there are at least $d+1$ coins that have the same weight. For this step, a total of $d+ 1 - (t - a \big\lfloor \frac{t}{a} \big\rfloor)$ coins are borrowed from the larger piles. Because $a \nmid d$, it follows that there must be $f$ fake coins.

For the given parameters, the set of admissible situations $W(\mathbf{s}|Z_f, \mathcal{A})$ includes situations in which all the fake coins are in the original $a$ piles, with each pile containing exactly $\frac{f}{a}$ fake coins. A simple calculation shows that these $a$ piles consist of $a\big\lfloor \frac{t}{a} \big\rfloor - a\big\lfloor \frac{d}{a} \big\rfloor +d + 1 -t$ piles of size $\big\lfloor \frac{t}{a} \big\rfloor - \big\lceil \frac{d}{a} \big\rceil$ in addition to $a - a\big\lfloor \frac{t}{a} \big\rfloor + a\big\lfloor \frac{d}{a} \big\rfloor -d - 1 +t$ piles of size $ \big\lfloor \frac{t}{a} \big\rfloor - \big\lfloor \frac{d}{a} \big\rfloor$.

\end{quote}

\begin{theorem}\label{thm:2}
If $a>1$ divides $f$ and neither $d$ nor $t$, $t>2a$, and $\big\lfloor \frac{t}{a} \big\rfloor - \big\lceil \frac{d}{a} \big\rceil > \frac{f}{a}$, then an optimal indiscreet weighing strategy must satisfy the following:

\begin{equation}\label{eq:3}
\dbinom{\big\lfloor \frac{t}{a} \big\rfloor - \big\lceil \frac{d}{a} \big\rceil}{\frac{f}{a}}^{a\big\lfloor \frac{t}{a} \big\rfloor - a\big\lfloor \frac{d}{a} \big\rfloor +d + 1 -t} \dbinom{\big\lfloor \frac{t}{a} \big\rfloor - \big\lfloor \frac{d}{a} \big\rfloor}{\frac{f}{a}}^{a - a\big\lfloor \frac{t}{a} \big\rfloor + a\big\lfloor \frac{d}{a} \big\rfloor -d - 1 +t} \le |W(\mathbf{s}|Z_f, \mathcal{A})|.
\end{equation}
\end{theorem}

\begin{proof}
The above WS is guaranteed to be successful for the given conditions. Furthermore, it follows that an optimal indiscreet strategy in the given scenarios must satisfy (\ref{eq:3}).
\end{proof}

It is worth noting that most of the time this strategy can be optimized to yield a result much better than the lower bound given in (\ref{eq:3}). For example, in the cases for which $d - a\big\lfloor \frac{d}{a} \big\rfloor > t -  a\big\lfloor \frac{t}{a} \big\rfloor$, or the remainder when $d$ is divided by $a$ is greater than the remainder when $t$ is divided by $a$, then $\mathcal{A}$ does not need to make comparisons between anything other than the $a$ large groups because the existence of exactly $d$ fake coins is impossible. In this case, the number of admissible situations after the series of weighings can be calculated exactly as $ |W(\mathbf{s}|Z_f, \mathcal{A})| = \dbinom{ \big\lfloor t/a \big\rfloor}{f/a} ^a$. As $t$ tends to infinity for such cases, the revealing factor approaches the same limit as (\ref{xinf}).

\section{A Conjecture on $|W(\mathbf{s}|Z_f, \mathcal{A})|$ for Optimal Discreet Strategies}\label{discreet}

The following is a generalization of Strategy~\ref{str:3}, and is successful and discreet for any initial set of parameters $t$, $f$, and $d$ that satisfy $\Bigl\lfloor \dfrac{t}{f} \Bigr\rfloor \ge 4 $ and $0 < d < f$; Theorem~\ref{easydisc}, first stated in \cite{NDTK}, follows immediately. We additionally use this strategy to conjecture a new lower bound on $|W(\mathbf{s}|Z_f, \mathcal{A})|$ for an optimal discreet $\mathcal{A}$ in a much broader case.

\begin{quote}
Let $t = f k + r$, where $k$ and $r$ are positive integers, $0 < r < f$, and $k \ge 4$. Begin by splitting the coins into $3f$ total groups: $A_1, A_2, \ldots , A_f \in \mathbold{A}$ , $B_1, B_2, \ldots , B_f \in \mathbold{B}$, and $C_1, C_2, \ldots , C_f \in \mathbold{C}$. The lawyer will put one fake coin in each of either $A_i$, $B_i$, or $C_i$, for $i = 1, 2, \ldots , f$.

In groups $A_i$ for $1 \le i \le r$, we will have $|A_i| = k - 2$, and for $r + 1 \le i \le f$ we will have $|A_i| = k - 3$. Similarly, we will have $|B_i| = 1$ for $1 \le i \le r$ and $|B_i| = 2$ otherwise, and in groups $C_i$ for $1 \le i \le r$ we will have $|C_i| = 2$ and $|C_i| = 1$ for all other values of $i$.

Now we carry out the weighings as follows. In $f - 1$ weighings, we show that the $k - 1$ coins from each $A_i + B_i$ for $1 \le i \le f$ balance one another on the scale. In $f - 1$ more weighings, we demonstrate that the 3 coins from each
$B_i + C_i$ for $1 \le i \le f$ are equal in weight.
\end{quote}

As shown in \cite{NDTK}, if $0 < d < f$ then the procedure of $\mathcal{A}$ guarantees that exactly one of the following cases is true:

\begin{onehalfspacing}

\begin{enumerate}
    \item each group $A_i \in \mathbold{A}$ contains one fake coin,
    \item each group $B_i \in \mathbold{B}$ contains one fake coin,
    \item each group $C_i \in \mathbold{C}$ contains one fake coin.
\end{enumerate}

\end{onehalfspacing}

\begin{theorem}\label{easydisc}
If $ \Bigl\lfloor \dfrac{t}{f} \Bigr\rfloor \ge 4 $ and $0<d<f$, there exists a successful discreet strategy.
\end{theorem}

Unfortunately, a problem arises when we consider values of $d > f$. For example, a quick check reveals that if we put one counterfeit coin in $B_1$ and place another counterfeit coin in each of $A_i$ and $C_i$ for $ 2 \le i \le f$, then $2f-1$ fake coins is an admissible situation for this checking procedure.

This strategy can be augmented by making all possible comparisons $A_w$ + $B_x$ = $A_y$ + $B_z$ and $B_w$ + $C_x$ = $B_y$ + $C_z$ on a scale with $w \not= y$, $x \not= z$, and each of $w,x,y,z$ located in either the closed interval $[1,r]$ or the closed interval $[r+1,f]$. We conjecture that these additional weighings ensure that no situation such as the one described above can exist as an admissible situation for this strategy for specific values of $d$. More specifically, we can conjecture the following:

\begin{conjecture}
If $ \Bigl\lfloor \dfrac{t}{f} \Bigr\rfloor \ge 4 $ and either $0<d<f$ or $d$ is not of the form $qf$ or $qf \pm r$ for some positive integer $q$, there exists a discreet strategy.
\end{conjecture}

Furthermore, if we let $\mathbold{\delta} = \{\mathbold{A},\mathbold{B},\mathbold{C}\}$ and let $\delta_i$ be the group of coins representing either $A_i$, $B_i$, or $C_i$, respectively, then it can be shown that this WS $\mathcal{A}$ gives us

$$ |W(\mathbf{s}|Z_f, \mathcal{A})| = \sum_{\mathbold{\delta}} \prod_{i=1}^f |\delta_i| = (k-2)^r(k-3)^{f-r} + 1^r \cdot 2^{f-r} + 2^r \cdot 1^{f-r},$$

or in terms of $t$ and $f$ exclusively,

\begin{equation}\label{eq:4}
|W(\mathbf{s}|Z_f, \mathcal{A})| =  \Bigl( \Bigl\lfloor \frac{t}{f} \Bigr\rfloor - 2 \Bigr)^{t- f \lfloor \frac{t}{f} \rfloor} \Bigl( \Bigl\lfloor \frac{t}{f} \Bigr\rfloor - 3 \Bigr)^{f-t + f \lfloor \frac{t}{f} \rfloor} + 2^{f-t + f \lfloor \frac{t}{f} \rfloor}  + 2^{t-f \lfloor \frac{t}{f} \rfloor}.
\end{equation}

If the above conjecture is true, then an optimal discreet weighing strategy for the given parameters must satisfy $|W(\mathbf{s}|Z_f, \mathcal{A})| \ge (\ref{eq:4})$. Regardless of the conjecture's validity, $|W(\mathbf{s}|Z_f, \mathcal{A})| \ge (\ref{eq:4})$ is satisfied for an optimal discreet strategy whenever $\Bigl\lfloor \dfrac{t}{f} \Bigr\rfloor \ge 4 $ and $0 < d < f$.

\section{A Discreet Weighing Strategy Using Linear Combinations}

Strategy \ref{str:6} enables us to show that the number of fake coins is a multiple of some positive integer $a > 1$. We would like to describe a new WS $\mathcal{A}$ that can demonstrate that the number of fake coins is some linear combination of conveniently chosen positive integers. That is, we want to show that $f$ can be represented in the form 

\begin{equation}\label{eq:6}
f=c_1x_1 + c_2x_2 + c_3x_3 + \ldots + c_kx_k
\end{equation}
for fixed positive integers $c_1,...,c_k\ge2$ and nonnegative integers $x_1,\dots,x_k$. We will call vector $\mathbf{x} = (x_1,\dots,x_k)$ a \textit{solution vector} if it satisfies equation (\ref{eq:6}).

\begin{strategy}\label{str:7} Organize the $t$ coins into $c_i$ groups of size $g_i$ with the same number of fake coins in each group for $1 \le i \le k$, so that the following conditions are met:

\begin{equation}\label{eq:7}
\sum_{i=1}^{k} c_i g_i = t \qquad \qquad 2 \le c_i \le f \qquad \qquad 1 \le g_i \le t/2
\end{equation}

and so that

\begin{equation}\label{eq:8}
\sum_{i=1}^{k} c_i x_i = f \qquad \qquad 0 \le x_i \le g_i
\end{equation}

is satisfied by at least one solution vector $\mathbf{x}$, and

\begin{equation}\label{eq:9}
\sum_{i=1}^{k} c_i x_i = d \qquad \qquad 0 \le x_i \le g_i
\end{equation} 
has no solution vector. 

The weighings are carried out as follows: compare all $c_i$ piles of size $g_i$ with each other on the balance scale for $1 \le i \le k$. Since both the number of total coins and the number of fake coins is the same in each of the $c_i$ groups of size $g_i$, all weighings will be balanced. This implies that the number of fake coins is of the form $\sum c_ix_i$ for $0 \le x_i \le g_i$, where each value of $x_i$ represents a possible number of fake coins in the group.

\end{strategy}

\begin{theorem}\label{lincombtheorem}
If conditions (\ref{eq:7}), (\ref{eq:8}), and (\ref{eq:9}) are satisfied, and if each value of $x_i$ for $1\le i \le k$ is both strictly greater than $0$ and strictly less than $g_i$ for at least one solution vector $\mathbf{x}$, then Strategy~\ref{str:7} satisfies the conditions in Definition~\ref{def:3} and Definition~\ref{def:4} for a successful, discreet strategy.
\end{theorem}

If the above strategy is discreet, the number of new possibilities for the distribution of the $f$ fake coins can be calculated as

\begin{equation}\label{eq:10}
    |W(\mathbf{s}|Z_f,\mathcal{A})| = \sum_{\mathbf{x}} \prod_{i=1}^{k} \dbinom{g_i}{x_i}^{c_i}.
\end{equation}

An efficient algorithm for finding every solution vector $\mathbf{x}$ that satisfies the above relations is given in \cite{algo}. Nevertheless, the problem of maximizing (\ref{eq:10}) subject to (\ref{eq:7}), (\ref{eq:8}), and (\ref{eq:9}) is still a difficult one, and may require brute force searches.

With $k=1$, $c_1=a$, and $g_1 = t/a$ for some integer $a > 1$ that divides both $f$ and $t$ but not $d$, notice that Strategy~\ref{str:7} is identical to Strategy~\ref{str:6}. However, with a slight adjustment, it is often possible to do significantly better than Strategy~\ref{str:6} given certain parameters, especially when one value of $c_j$ can be written as a linear combination of the other values of $c_i$ for $i \neq j$. For example, consider the case of $t=70$, $f=7$, and $d=1$. Using Strategy~\ref{str:6} with $a=7$, we have $|W(\mathbf{s}|Z_f,\mathcal{A})|=\binom{10}{1}^7$ and $X \approx 119.88$. Using Strategy~\ref{str:7} with $(c_1,c_2,c_3)=(2,2,3)$ and $(g_1,g_2,g_3)=(10,10,10)$, we have three solution vectors $\mathbf{x} = (1,1,1)$, $(2,0,1)$, and $(0,2,1)$. This gives us $|W(\mathbf{s}|Z_f,\mathcal{A})|=\binom{10}{1}^7 + \binom{10}{1}^3\binom{10}{2}^2+\binom{10}{1}^3\binom{10}{2}^2$ and $X \approx 85.32$, a marked improvement even though both strategies involved splitting up the $70$ coins into $7$ smaller piles of size $10$.

\section{A New Variety of the Original Counterfeit Coin Problem}

Motivated by the results found in this work, we can describe a new variation of the original counterfeit coin problem:

\begin{quote}
Given a set of $n$ coins with some number $0 \le k \le n$ fake coins, use a balance scale to determine whether or not $k$ is a multiple of positive integer $m>1$.
\end{quote}

The following are a definition and theorem proposed by Purdy in \cite{purdy1}:

\begin{definition}
Let $f: \{0,1 \}^n \rightarrow \mathbb{R}$ be a function of $n$ Boolean variables, and let $x \in \{0,1 \}^n $. For $1 \le i \le n$, let $x^i$ be $x$ with its $i$-th coordinate flipped. The \textit{sensitivity of $f$ at $x$}, $\sigma_x(f)$, is the number of $i$ such that $f(x^i) \neq f(x)$. The \textit{average sensitivity of $f$}, $\alpha(f)$, is $E_x[\sigma_x(f)]$, where  $x \in \{0,1 \}^n $ is chosen uniformly at random.
\end{definition}

\begin{theorem}\label{thm:6}
Let $f: \{0,1\}^n \rightarrow \mathbb{R}$ be a function of $n$ Boolean variables. We can identify any sets of coins $C$ with a $0/1$ vector $x$ by numbering the coins and letting $x_i=1$ iff the $i$-th coin is good. We can then try to determine $f(x)$ by applying measurements to $C$. Any oblivious coin-weighing algorithm for determining $f$ in this way must use $\Omega(\alpha(f)/\sqrt{n})$ measurements, where $\alpha(f)$ is the average sensitivity of $f$.
\end{theorem}

Using the above definition of average sensitivity with Purdy's main result, along with the convention that $0$ indicates a fake coin and $1$ indicates a good (\textit{i.e.} real) coin, we can attempt to find the minimum number of measurements necessary to solve our new counterfeit coin problem. First, we define an appropriate measurement function $\text{MOD}_m^{*}$:

\begin{definition}
The function $\text{MOD}_m: \{0,1 \}^n \rightarrow \{0,1 \}$ is a function of $n$ Boolean variables that outputs 1 if the number of input 1s is a multiple of $m$, and 0 otherwise. We define $\text{MOD}_m^{*}$ as a Boolean function that outputs 1 if the number of input 0s is a multiple of $m$, and 0 otherwise.
\end{definition}

If we let the function $\text{MOD}_m^{*}$ to be our measurement function for the given coin weighing problem, it is easy to see that a result of $\text{MOD}_m^{*}=1$ implies that $k$ is a multiple of $m$, and $\text{MOD}_m^{*}=0$ implies the contrary. If we can find the average sensitivity of $\text{MOD}_m^{*}$, Theorem~\ref{thm:6} can be applied to show the minimum number of weighings necessary to solve the problem.

\begin{theorem}\label{thm:7}
The Boolean function $\text{MOD}_m^{*}: \{0,1 \}^n \rightarrow \{0,1 \}$ has average sensitivity $\dfrac{2n}{m}+ o(1)$.
\end{theorem}

\begin{proof}
According to the definition of average sensitivity, we have 
$$ \alpha(\text{MOD}_m^{*})= E_x[\sigma_x(\text{MOD}_m^{*})] $$
\begin{equation}\label{eq:11}
 = \frac{1}{2^n} \sum_{x \in \{0,1 \}^n} \sigma_x(\text{MOD}_m^{*}).
\end{equation}

The above sum can be split up into four separate cases for all binary strings of length $n$:

\begin{enumerate}
    \item The number of $0$'s is $sm$ for some $s$. The number of such strings is $\binom{n}{sm}$. In this case, flipping any of the bits produces a change in the output of $\text{MOD}_m^{*}$, so the sensitivity is $n$.
    \item The number of $0$'s is $sm-1$ for some $s$. The number of such strings is $\binom{n}{sm-1}$. In this case, flipping any of the $(n-sm+1)$ $1$'s produces a change in the output of $\text{MOD}_m^{*}$, so the sensitivity is $n-sm+1$.
    \item The number of $0$'s is $sm+1$ for some $s$. The number of such strings is $\binom{n}{sm+1}$. In this case, flipping any of the $(sm+1)$ $0$'s produces a change in the output of $\text{MOD}_m^{*}$, so the sensitivity is $sm+1$.
    \item In all other strings, changing any single bit produces no change in the output of $\text{MOD}_m^{*}$, so the sensitivity is $0$.
\end{enumerate}

As a result, our expression in (\ref{eq:11}) becomes

$$ \frac{1}{2^n} \sum_{s=0}^{\lfloor n/m \rfloor} \binom{n}{sm-1}(n-sm+1) + \binom{n}{sm}n + \binom{n}{sm+1}(sm+1) + 0 $$  
$$ = \frac{n}{2^{n-1}} \sum_{s=0}^{\lfloor n/m \rfloor} \binom{n}{sm} $$
\begin{equation}\label{eq:12}
= \frac{n}{2^{n-1}} \sum_{s=0}^{\infty} \binom{n}{sm}.
\end{equation}

As shown in \cite{sms}, series multisection results in the well-known identity

\begin{equation}\label{eq:13}
\sum_{s=0}^\infty \binom{n}{p+sm} = \frac{1}{m} \sum_{j=0}^{m-1} \cos \bigg[ \dfrac{\pi (n-2p)j}{m} \bigg] 2^n \cos^n \bigg (\dfrac{\pi j}{m} \bigg).
\end{equation}

By substitution, equation (\ref{eq:12}) can be rewritten as

$$ \frac{n}{2^{n-1}} \sum_{s=0}^{\infty} \binom{n}{sm} = \frac{n}{2^{n-1}} \Bigg[ \frac{1}{m} \sum_{j=0}^{m-1} \cos \bigg( \dfrac{\pi nj}{m} \bigg) 2^n \cos^n \bigg (\dfrac{\pi j}{m} \bigg) \Bigg] $$

$$ = \frac{2n}{m} \Bigg[\sum_{j=0}^{m-1} \cos \bigg( \dfrac{\pi nj}{m} \bigg) \cos^n \bigg (\dfrac{\pi j}{m} \bigg) \Bigg] $$

\begin{equation}\label{eq:13.5}
= \frac{2n}{m} \Bigg[1 + \sum_{j=1}^{m-1} \cos \bigg( \dfrac{\pi nj}{m} \bigg) \cos^n \bigg (\dfrac{\pi j}{m} \bigg) \Bigg].
\end{equation}

By analysis of each term in the sum of the above expression, it is clear that 

$$ \sum_{j=1}^{m-1} \cos \bigg( \dfrac{\pi nj}{m} \bigg) \cos^n \bigg (\dfrac{\pi j}{m} \bigg) \le \sum_{j=1}^{m-1} \cos^n \bigg (\dfrac{\pi j}{m} \bigg) \le  \sum_{j=1}^{m-1} \cos^n \bigg (\dfrac{\pi}{m} \bigg) < m(1-\epsilon)^n, $$

where $\epsilon$ is a positive constant such that $\epsilon < 1 - \cos (\frac{\pi}{m})$. This in turn implies that our expression in (\ref{eq:13.5}) can be written as 

$$ \frac{2n}{m} [1 + o(m(1-\epsilon)^n)] $$

$$ = \frac{2n}{m} + o(n(1-\epsilon)^n) $$

\begin{equation}\label{eq:14}
= \frac{2n}{m} + o(1),
\end{equation}

as desired.

\end{proof}

It is easy to make sense of this result. Suppose $n \gg m$. In this case, there is approximately a $1/m$ chance that the number of fake coins is a multiple of $m$; in this case, flipping the identity of any of the coins changes the output of $\text{MOD}_m^{*}$. In addition, there’s about a $1/m$ chance that the number of fake coins is one less than a multiple of $m$, and a $1/m$ chance that this number is one more than a multiple of $m$; in both cases, flipping the identity of one of the coins changes the output around 50\% of the time. All in all, you get the probability of a changed outcome as $(1/m)(1) + (1/m)(1/2) + (1/m)(1/2) = 2/m$. Multiplying this value by $n$, the total number of coins, gives us $2n/m$ as expected.

Combining Theorems \ref{thm:6} and \ref{thm:7} and substituting $\text{MOD}_m^{*}$ for $f$, we may state the following:

\begin{theorem}
Given a set of $n$ total coins with $0 \le k \le n$ fake, at least $\Omega(\frac{\sqrt{n}}{m}+ o(1))$ measurements are required for an oblivious coin-weighing algorithm to determine if $k$ is a multiple of $m$.
\end{theorem}

\section{Conclusions and Future Work}

We have made several advances in the analysis of this coin weighing problem, but many questions still remain unanswered: do efficient methods of maximizing the results given by specific strategies, such as equation (\ref{eq:10}), exist? Can we construct an optimal weighing strategy $\mathcal{A}$ for any initial set of parameters, and if so, how? Furthermore, this paper primarily examines strategies that utilize balanced weighings, and the generalization of weighing procedures that use unbalanced weighings, such as Strategy~\ref{str:1} and others shown in \cite{Knop, NDTK}, has not even been addressed. Additionally, this problem would benefit greatly from a formal treatment with information theory; for instance, there seems to be a very strong correlation between $\log_3(X)$ and the number of weighings required to carry out a WS $\mathcal{A}$, but no formal relation has yet been discerned. By formalizing the definition of this relatively new problem and generalizing a number of solutions, we have laid the groundwork for future research that can tackle these issues. A more thorough analysis of this problem may prove to be useful in fields such as cryptology and information science.

\section{Acknowledgments}

I am grateful to the MIT-PRIMES program for allowing me to conduct this research, and to Tanya Khovanova for introducing me to this problem and being my mentor. Rafael M. Saavedra's useful advice and suggestions are also highly appreciated.

\end{document}